\documentclass[11pt,A4]{amsart}
\usepackage{ifthen,latexsym,amssymb,amsmath}
\newtheorem{theorem}{Theorem}[section]
\newtheorem{proposition}[theorem]{Proposition}
\newtheorem{lemma}[theorem]{Lemma}
\newtheorem{corollary}[theorem]{Corollary}

\newcommand{\sumc}{\sum_\text{cycl}}
\newcommand{\sump}{\sum_\text{perm}}
\newcommand{\cS}[1]{\sigma^{\text{cycl}}_{#1}}
\newcommand{\pS}[1]{\sigma^{\text{perm}}_{#1}}
\newcommand{\oct}{\text{oct}}

\begin{document}

\newcommand{\sgn}{\text{sgn}}

\author{Bart\l{}omiej Bzd\c{e}ga}

\address{Adam Mickiewicz University, Pozna\'n, Poland}

\email{exul@amu.edu.pl}



\title{Jumps of ternary cyclotomic coefficients}

\maketitle

\begin{abstract}
It is known that two consecutive coefficients of a ternary cyclotomic polynomial $\Phi_{pqr}(x)=\sum_k a_{pqr}(k)x^k$ differ by at most one. In this paper we give a criterion on $k$ to satisfy $|a_{pqr}(k)-a_{pqr}(k-1)|=1$. We use this to prove that the number of nonzero coefficients of the $n$th ternary cyclotomic polynomial is greater than $n^{1/3}$.
\end{abstract}

\section{Introduction}

We define the $n$th cyclotomic polynomial in the following way
$$\Phi_n(x) = \prod_{1\le m\le n,\;(m,n)=1}(x-\zeta_n^m) = \sum_{k\in\mathbb{Z}}a_n(k)x^k, \quad \text{where }\zeta_n=e^{2\pi i/n}.$$
We say that a polynomial $\Phi_n$ is binary if $n$ is a product of two distinct odd primes, ternary if $n$ is a product of three distinct odd primes, etc.

The coefficients of cyclotomic polynomials are an interesting object to study. One of the most intensively studied research directions is estimating the maximal absolute value of a coefficient of $\Phi_n$ \cite{Bachman-TernaryBounds,BatemanPomeranceVaughan-Size,Beiter-TernaryI,Beiter-TernaryII,Bzdega-Ternary,Bzdega-Height}. There are also papers on the sum of the absolute values of coefficients of $\Phi_n$ \cite{BatemanPomeranceVaughan-Size,Bzdega-Height} and on the number $\theta_n$ of its nonzero coefficients \cite{Bzdega-Sparse,Carlitz-Terms}.

Ternary cyclotomic polynomials have an interesting property discovered by Gallot and Moree \cite{GallotMoree-JumpOne}: the difference between $a_{pqr}(k)$ and $a_{pqr}(k-1)$ never exceeds $1$. In this paper, for a given ternary cyclotomic polynomial $\Phi_{pqr}$, we characterize all $k$ such that $|a_{pqr}(k)-a_{pqr}(k-1)|=1$. Also we determine the number of $k$'s for which this equality holds.

We say that the coefficient $a_{pqr}(k)$ is jumping up if $a_{pqr}(k)=a_{pqr}(k-1)+1$. Analogously we define jumping down coefficients. Cyclotomic polynomials are known to be palindromic, i.e. $a_n(k)=a_n(\varphi(n)-k)$, where $\varphi(n)$ is the Euler function and the degree of $\Phi_n$. Therefore the number of jumping up and the number of jumping down coefficients are equal and we denote this number by $J_{pqr}$.

One of our main results is the following theorem.

\begin{theorem} \label{thm-lb}
For a ternary cyclotomic polynomial $\Phi_n$ we have $J_n>n^{1/3}$.
\end{theorem}

One half of the total number of jumping (up or down) coefficients is the lower bound for the number of odd coefficients of $\Phi_{pqr}$ and thus it is the lower bound for the number $\theta_{pqr}$ of nonzero coefficients of $\Phi_{pqr}$. So we have

\begin{corollary} \label{cor-n0}
Let $\Phi_n$ be a ternary cyclotomic polynomial. Then $\theta_n>n^{1/3}$.
\end{corollary}

We do not know if for every $\epsilon>0$ there exist infinite classes of ternary cyclotomic polynomials $\Phi_n$ with $J_n<n^{1/3+\varepsilon}$. However, under some strong assumptions, we can prove that they do.

\begin{theorem} \label{thm-smallJ}
Let $\varepsilon>0$. If $q$ is a Germain prime, $q+1$ has a prime divisor $p>q^{1-\varepsilon}$ and $r=2q+1$, then $J_n<10n^{1/(3-\varepsilon)}$, where $n=pqr$.
\end{theorem}

If the celebrated Schinzel Hypothesis H is true then there exist infinitely many triples of primes $(p,q,r)$ satisfying conditions of Theorem \ref{thm-smallJ}. For example, we can put $(p,q,r)=(m,6m-1,12m-1)$.

The paper is organized in the following way. In the second section we recall some results form our earlier work \cite{Bzdega-Ternary}. In the third section we give a criterion on $k$ determining $V(k)=a_{pqr}(k)-a_{pqr}(k-1)\in\{-1,0,1\}$. In the fourth section we derive a formula on $J_{pqr}$ and prove Theorem \ref{thm-lb}. In the fifth section we prove Theorem \ref{thm-smallJ} and discuss the case of inclusion-exclusion polynomials.

\section{Preliminaries}

Troughout the paper we fix distinct odd primes $p$, $q$, $r$. Let us emphasize that every fact we prove for $(p,q,r)$ has also an appropriate symmetric versions.

By $a^{-1}(b)$ we denote the inverse of $a$ modulo $b$ for $(a,b)=1$. We treat this number as an integer from the set $\{1,2,\ldots,b-1\}$.

For every integer $k$ we define $F_k\in\mathbb{Z}$ and $a_k\in\{0,1,\ldots,p-1\}$, $b_k\in\{0,1,\ldots,q-1\}$, $c_k\in\{0,1,\ldots,r-1\}$ by the equation
$$k+F_kpqr = a_kqr+b_krp+c_kpq,$$
which clearly has the unique solution $(F_k,a_k,b_k,c_k)$ depending on $k$. In \cite{Bzdega-Ternary} we proved the following properties of numbers $F_k$.

\begin{proposition}[\cite{Bzdega-Ternary}, a remark before Lemma 2.1] \label{pre-012}
For $-(qr+rp+pq)<k<pqr$ we have $F_k\in\{0,1,2\}$.
\end{proposition}

\begin{proposition}[\cite{Bzdega-Ternary}, Lemma 2.2] \label{pre-2}
We have
$$F_k-F_{k-q} = \left\{\begin{array}{rl}
-1, & \text{if } a_k<r^{-1}(p) \text{ and } c_k<p^{-1}(r), \\
1, & \text{if } a_k\ge r^{-1}(p) \text{ and } c_k\ge p^{-1}(r), \\
0, & \text{otherwise.}
\end{array}\right.$$
\end{proposition}

\begin{proposition}[\cite{Bzdega-Ternary}, Lemma 2.3] \label{pre-4}
We have
$$F_k - F_{k-q} - F_{k-r} + F_{k-q-r} = \left\{\begin{array}{rl}
-1, & \text{if } a_k \in \mathcal{A}^p_1, \\
1, & \text{if } a_k \in \mathcal{A}^p_3, \\
0, & \text{otherwise,}
\end{array}\right.$$
where
\begin{align*}
\mathcal{A}^p_1 & = \{0,1,\ldots,p-1\}\cap\big[q^{-1}(p)+r^{-1}(p)-p,\min\{q^{-1}(p),r^{-1}(p)\}\big), \\
\mathcal{A}^p_3 & = \{0,1,\ldots,p-1\}\cap\big[\max\{q^{-1}(p),r^{-1}(p)\},q^{-1}(p)+r^{-1}(p)\big).
\end{align*}
\end{proposition}

\begin{proposition}[\cite{Bzdega-Ternary}, Lemma 2.4] \label{pre-8}
We have
$$F_k - F_{k-p} - F_{k-q} - F_{k-r} + F_{k-q-r} + F_{k-r-p} + F_{k-p-q} - F_{k-p-q-r} = 0.$$
\end{proposition}

\begin{proposition}[\cite{Bzdega-Ternary}, Lemma 5.1] \label{pre-jump}
For $k=0,1,\ldots,pqr-1$ we have
\begin{align*}
& a_{pqr}(k)-a_{pqr}(k-1) \\
& = N_0(F_k,F_{k-q-r},F_{k-r-p},F_{k-p-q}) - N_0(F_{k-p},F_{k-q},F_{k-r},F_{k-p-q-r}) \\
& = N_2(F_k,F_{k-q-r},F_{k-r-p},F_{k-p-q}) - N_2(F_{k-p},F_{k-q},F_{k-r},F_{k-p-q-r}) \\
& = \frac12\big(N_1(F_{k-p},F_{k-q},F_{k-r},F_{k-p-q-r}) - N_1(F_k,F_{k-q-r},F_{k-r-p},F_{k-p-q})\big),
\end{align*}
where $N_t(s)$ denotes the number of $t$'s in the sequence $(s)$.
\end{proposition}

\section{A criterion on jumping coefficients}

We define five following sets:
\begin{align*}
\mathcal{A}^p_0 & = \{0,1,\ldots,p-1\}\cap\big[0,q^{-1}(p)+r^{-1}(p)-p\big), \\
\mathcal{A}^p_1 & = \{0,1,\ldots,p-1\}\cap\big[q^{-1}(p)+r^{-1}(p)-p, \min\{q^{-1}(p),r^{-1}(p)\}\big), \\
\mathcal{A}^p_2 & = \{0,1,\ldots,p-1\}\cap\big[\min\{q^{-1}(p),r^{-1}(p)\},\max\{q^{-1}(p),r^{-1}(p)\}\big), \\
\mathcal{A}^p_3 & = \{0,1,\ldots,p-1\}\cap\big[\max\{q^{-1}(p),r^{-1}(p)\},q^{-1}(p)+r^{-1}(p)\big), \\
\mathcal{A}^p_4 & = \{0,1,\ldots,p-1\}\cap\big[q^{-1}(p)+r^{-1}(p),p\big).
\end{align*}
Similarly we define $\mathcal{A}^q_j$ and $\mathcal{A}^r_j$ for $j=0,1,2,3,4$.

By Preposition \ref{pre-jump}, we have to consider $8$-tuples
$$\oct(k) = (F_k, F_{k-p}, F_{k-q}, F_{k-r}, F_{k-q-r}, F_{k-r-p}, F_{k-p-q}, F_{k-p-q-r}).$$
We write $\oct(k)\sim(t_1,\ldots,t_8)$ if $\oct(k)=(t_1+u,\ldots,t_8+u)$ for some integer $u$. Put also
$$V(k) = a_{pqr}(k)-a_{pqr}(k-1)$$
and
$$\delta_{qr}=\left\{\begin{array}{rl}
1, & \text{if } q^{-1}(p) < r^{-1}(p), \\
0, & \text{otherwise.}
\end{array}\right.$$
Analogously we define $\delta_{rq}$, $\delta_{rp}$, $\delta_{pr}$, $\delta_{pq}$ and $\delta_{pq}$.

The following theorem derives a criterion for the $k$th coefficient of $\Phi_{pqr}$ to be jumping up or down.

\pagebreak

\begin{theorem} \label{thm-criterion}
The value $V(k)$ depends on which one of the sets $\mathcal{A}^p_{j_1}\times\mathcal{A}^q_{j_2}\times\mathcal{A}^r_{j_3}$ contains $(a_k,b_k,c_k)$ in the way described in Table \ref{tab-oct}. The notation $(j_1j_2j_3)$ in the first column means $(a_k,b_k,c_k)\in\mathcal{A}^p_{j_1}\times\mathcal{A}^q_{j_2}\times\mathcal{A}^r_{j_3}$.
\end{theorem}

\begin{table}[!htbp] \caption{the values of $\oct(k)$ in dependence on $(a_k,b_k,c_k)$} \label{tab-oct}
\begin{center}\begin{tabular}{|c|c|c|} \hline
$(j_1j_2j_3)$ & $\oct(k) \sim$ & $V(k)$ \\ \hline
$001$ & $(0,1,1,1,1,2,2,2)$ & $1$ \\ \hline
$002$ & $(0,\delta_{pq},\delta_{qp},1,1+\delta_{qp},1+\delta_{pq},1,2)$ & $0$ \\ \hline
$003$ & $(0,0,0,1,1,1,1,2)$ & $-1$ \\ \hline
$004$ & $(0,0,0,1,1,1,0,1)$ & $0$ \\ \hline
$011$ & $(0,1,1,1,2,1,1,1)$ & $1$ \\ \hline
$012$ & $(0,\delta_{pq},\delta_{qp},1,1+\delta_{qp},\delta_{pq},1,1)$ & $\delta_{qp}$ \\ \hline
$013$ & $(0,0,0,1,1,0,1,1)$ & $0$ \\ \hline
$014$ & $(0,0,0,1,1,0,0,0)$ & $0$ \\ \hline
$022$ & $(0,\delta_{pq}+\delta_{pr}-1,\delta_{qp},\delta_{rp},\delta_{qp}+\delta_{rp},\delta_{pq},\delta_{pr},1)$ & $0$ \\ \hline
$023$ & $(1,\delta_{pr},1,1+\delta_{pr},1+\delta_{pr},1,1+\delta_{pr},2)$ & $-\delta_{rp}$ \\ \hline
$024$ & $(1,\delta_{pr},1,1+\delta_{pr},1+\delta_{pr},1,\delta_{pr},1)$ & $0$ \\ \hline
$033$ & $(1,0,1,1,1,1,1,2)$ & $-1$ \\ \hline
$034$ & $(1,0,1,1,1,1,0,1)$ & $0$ \\ \hline
$044$ & $(1,0,1,1,1,0,0,0)$ & $0$ \\ \hline
$111$ & $(0,1,1,1,1,1,1,0)$ & $0$ \\ \hline
$112$ & $(0,\delta_{pq},\delta_{qp},1,\delta_{qp},\delta_{pq},1,0)$ & $0$ \\ \hline
$113$ & $(0,0,0,1,0,0,1,0)$ & $0$ \\ \hline
$114$ & $(1,1,1,2,1,1,1,0)$ & $-1$ \\ \hline
$122$ & $(1,\delta_{pq}+\delta_{pr},1+\delta_{qp},1+\delta_{rp},$ & $\delta_{pq}\delta_{pr}-\delta_{qp}\delta_{rp}$ \\
      & $\delta_{qp}+\delta_{rp},1+\delta_{pq},1+\delta_{pr},1)$ & \\ \hline
$123$ & $(1,\delta_{pr},1,1+\delta_{rp},\delta_{rp},1,1+\delta_{pr},1)$ & $\delta_{pr}-\delta_{rp}$\\ \hline
$124$ & $(1,\delta_{pr},1,1+\delta_{rp},\delta_{rp},1,\delta_{pr},0)$ & $-\delta_{rp}$\\ \hline
$133$ & $(1,0,1,1,0,1,1,1)$ & $0$ \\ \hline
$134$ & $(1,0,1,1,0,1,0,0)$ & $0$ \\ \hline
$144$ & $(2,1,2,2,1,1,1,0)$ & $-1$ \\ \hline
$222$ & $(1,\delta_{pq}+\delta_{pr},\delta_{qr}+\delta_{qp},\delta_{rp}+\delta_{rq},$ & $0$ \\
      & $\delta_{qp}+\delta_{rp},\delta_{rq}+\delta_{pq},\delta_{pr}+\delta_{qr},1)$ & \\ \hline
$223$ & $(1,\delta_{pr},\delta_{qr},\delta_{rp}+\delta_{rq},\delta_{rp},\delta_{rq},\delta_{pr}+\delta_{qr},1)$ & $\delta_{pr}\delta_{qr}-\delta_{rp}\delta_{rq}$ \\ \hline
$224$ & $(1,\delta_{pr},\delta_{qr},\delta_{rp}+\delta_{rq},\delta_{rp},\delta_{rq},\delta_{pr}+\delta_{qr}-1,0)$ & $0$ \\ \hline
$233$ & $(1,0,\delta_{qr},\delta_{rq},0,\delta_{rq},\delta_{qr},1)$ & $0$ \\ \hline
$234$ & $(2,1,1+\delta_{qr},1+\delta_{rq},1,1+\delta_{rq},\delta_{qr},1)$ & $\delta_{rq}$ \\ \hline
$244$ & $(2,1,1+\delta_{qr},1+\delta_{rq},1,\delta_{rq},\delta_{qr},0)$ & $0$ \\ \hline
$333$ & $(1,0,0,0,0,0,0,1)$ & $0$ \\ \hline
$334$ & $(2,1,1,1,1,1,0,1)$ & $1$ \\ \hline
$344$ & $(2,1,1,1,1,0,0,0)$ & $1$ \\ \hline
\end{tabular}\end{center}\end{table}

In order to prove Theorem \ref{thm-criterion} we need the following, simple fact.

\begin{lemma} \label{lem-help}
If $(a_k,b_k,c_k)\in\mathcal{A}^p_{j_1}\times\mathcal{A}^q_{j_2}\times\mathcal{A}^r_3$ and $(a_{k'},b_{k'},c_{k'})\in\mathcal{A}^p_{j_1}\times\mathcal{A}^q_{j_2}\times\mathcal{A}^r_4$, then
$$\oct(k') \sim (F_k, F_{k-p}, F_{k-q}, F_{k-r}, F_{k-q-r}, F_{k-r-p}, F_{k-p-q}+1, F_{k-p-q-r}+1).$$
Similarly, if $k\in\mathcal{A}^p_1\times\mathcal{A}^q_{j_2}\times\mathcal{A}^r_{j_3}$ and $k'\in\mathcal{A}^p_0\times\mathcal{A}^q_{j_2}\times\mathcal{A}^r_{j_3}$, then
$$\oct(k') \sim (F_k, F_{k-p}, F_{k-q}, F_{k-r}, F_{k-q-r}-1, F_{k-r-p}, F_{k-p-q}, F_{k-p-q-r}-1).$$
\end{lemma}

\begin{proof}
Let us consider the first situation. By Proposition \ref{pre-2} and its symmetric versions, it follows that
$$F_k-F_{k-p}=F_{k'}-F_{k'-p}, \quad F_k-F_{k-q}=F_{k'}-F_{k'-q}, \quad F_k-F_{k-r}=F_{k'}-F_{k'-r}.$$
Then, by Proposition \ref{pre-4} and its symmetric versions we have
$$F_k-F_{k-q-r} = F_{k'}-F_{k'-q-r}, \qquad F_k-F_{k-r-p} = F_{k'}-F_{k'-r-p}.$$
Once more we use a symmetric version of Proposition \ref{pre-4}, which gives the equalities
\begin{align*}
F_k-F_{k-p}-F_{k-q}+F_{k-p-q} = F_{k-r}-F_{k-r-p}-F_{k-q-r}+F_{k-p-q-r} & = 1, \\
F_{k'}-F_{k'-p}-F_{k'-q}+F_{k'-p-q} = F_{k'-r}-F_{k'-r-p}-F_{k'-q-r}+F_{k'-p-q-r} & = 0.
\end{align*}
Thus the first claim is true. The proof of the second one is similar.
\end{proof}

Now we are ready to prove the main result of this section.

\begin{proof}[Proof of Theorem \ref{thm-criterion}]
We determine $\oct(k)$ up to adding an integer, so in each row of Table \ref{tab-oct} we fixed $F_k$ arbitrarily. First, consider the cases $(j_1j_2j_3)$ from Table \ref{tab-oct} which are not of the form $(0\ldots)$ or $(\ldots4)$. Using Proposition \ref{pre-2} and its symmetric versions, we obtain the values of $F_{k-p}$, $F_{k-q}$, $F_{k-r}$. Then, by Proposition \ref{pre-4} and its symmetric versions we compute $F_{k-q-r}$, $F_{k-r-p}$, $F_{k-p-q}$. At the end we use Proposition \ref{pre-8} to determine $F_{k-p-q-r}$.

We assumed that $\delta_{pq}+\delta_{qp}=1$ since if $\delta_{pq}=\delta_{qp}=0$ then $\mathcal{A}^r_2=\emptyset$ and the case is empty. The situation with $\delta_{rq}$, $\delta_{rq}$ and $\delta_{rp}$, $\delta_{pr}$ is analogous.

Now we can use Lemma \ref{lem-help} to compute $\oct(k)$ for remaining cases $(j_1j_2j_3)$: of form $(0\ldots)$ and $(\ldots4)$. After these computations the second column of Table \ref{tab-oct} is complete.

It is time to compute $V(k)$, for which we use Proposition \ref{pre-jump}. In rows which does not contain $\delta$'s the calculation is straightforward. We consider remaining cases one by one. We write $\oct(k)\sim(\ldots)$ if there is the equality up to adding an integer.
\begin{itemize}
\item[(002)] It does not matter which one of $\delta_{pq}$, $\delta_{qp}$ equals $1$, so we assume that $\delta_{pq}=1$. Then $\oct(k)=(0,1,0,1,1,2,1,2)$ and $V(k)=0$.
\item[(244)] This case is analogous to the previous one.
\item[(112)] Again, the value of $\delta_{pq}$ is not important and for $\delta_{pq}=1$ we have $\oct(k)\sim(0,1,0,1,0,1,1,0)$ and $V(k)=0$.
\item[(233)] As before, the value of $\delta_{qr}$ has no influence on $V(k)$ and we can assume $\delta_{qr}=1$. Then $\oct(k)\sim(1,0,1,0,0,0,1,1)$ and $V(k)=0$.
\item[(012)] For $\delta_{pq}=1$ we have $\oct(k)\sim(0,1,0,1,1,1,1,1)$ and $V(k)=0$. For $\delta_{qp}=1$ we have $\oct(k)=(0,0,1,1,2,0,1,1)$ and $V(k)=1$. Thus $V(k)=\delta_{qp}$.
\item[(234)] For $\delta_{qr}=1$ we have $\oct(k)\sim(2,1,2,1,1,1,1,1)$ and $V(k)=0$. For $\delta_{rq}=1$ we have $\oct(k)=(2,1,1,2,1,2,0,1)$ and $V(k)=1$. So $V(k)=\delta_{rq}$.
\item[(024)] If $\delta_{pr}=0$ then $\oct(k)\sim(1,0,1,1,1,1,0,1)$. If $\delta_{pr}=1$ then $\oct(k)\sim(1,1,1,2,2,1,1,1)$. In both cases $V(k)=0$.
\item[(123)] If $\delta_{rp}=1$, then $\oct(k)=(1,0,1,2,1,1,1,1)$ and $V(k)=-1$. If $\delta_{pr}=1$, then $\oct(k)=(1,1,1,1,0,1,2,1)$ and $V(k)=1$. So we have $V(k)=\delta_{pr}-\delta_{rp}$.
\item[(023)] For $\delta_{rp}=1$ we have $\oct(k)=(1,0,1,1,1,1,1,2)$ and $V(k)=-1$. For $\delta_{pr}=1$ we have $\oct(k)\sim(1,1,1,2,2,1,2,2)$ and $V(k)=0$. Thus $V(k)=-\delta_{rp}$.
\item[(124)] If $\delta_{rp}=1$, then $\oct(k)=(1,0,1,2,1,1,0,0)$ and $V(k)=-1$. When $\delta_{pr}=1$, we have $\oct(k)\sim(1,1,1,2,2,1,1,1)$ and $V(k)=0$. So $V(k)=-\delta_{rp}$.
\item[(022)] Note that in this case the situation $\delta_{pq}=\delta_{pr}=0$ is impossible, because then $F_{m-q-r}=F_{m-p}+3$, contradicting Proposition \ref{pre-012}. If $\delta_{pq}=\delta_{pr}=1$, then $\oct(k)\sim(0,1,0,0,0,1,1,1)$ and $V(k)=0$. If one of $\delta_{pq}$, $\delta_{pr}$ equals $1$, we assume that $\delta_{pq}=1$ (it does not matter). Then we have $\oct(k)\sim(0,0,0,1,1,1,0,1)$ and $V(k)=0$.
\item[(224)] This case is analogous to the previous one.
\item[(122)] We have the following equalities
$$\oct(k)\sim\left\{\begin{array}{rl}
(1,2,1,1,0,2,2,1) & \text{if } \delta_{pq}=\delta_{pr}=1, \\
(1,0,2,2,2,1,1,1) & \text{if } \delta_{qp}=\delta_{rp}=1, \\
(1,1,1,2,1,2,1,1) & \text{if } \delta_{pq}=\delta_{rp}=1, \\
(1,1,2,1,1,1,2,1) & \text{if } \delta_{qp}=\delta_{pr}=1,
\end{array}\right.$$
by which we obtain $V(k)=\delta_{pq}\delta_{pr}-\delta_{qp}\delta_{rp}$.
\item[(223)] Similarly as in the previous case, we have
$$\oct(k)\sim\left\{\begin{array}{rl}
(1,1,1,0,0,0,2,1) & \text{if } \delta_{pr}=\delta_{qr}=1, \\
(1,0,0,2,1,1,0,1) & \text{if } \delta_{rp}=\delta_{rq}=1, \\
(1,1,0,1,0,1,1,1) & \text{if } \delta_{pr}=\delta_{rq}=1, \\
(1,0,1,1,1,0,1,1) & \text{if } \delta_{rp}=\delta_{qr}=1.
\end{array}\right.$$
We conclude that $V(k)=\delta_{pr}\delta_{qr}-\delta_{rp}\delta_{rq}$.
\item[(222)] We have
$$\oct(k)\sim\left\{\begin{array}{rl}
(1,1,1,1,1,1,1,1) & \text{if } \delta_{pq}=\delta_{qr}=\delta_{rp}=1, \\
(1,2,1,0,0,1,2,1) & \text{if } \delta_{pq}=\delta_{qr}=\delta_{pr}=1.
\end{array}\right.$$
In both cases above we have $V(k)=0$. The remaining ones are symmetric.
\end{itemize}
Thus we verified all the cases from Table \ref{tab-oct} and the proof of Theorem \ref{thm-criterion} is complete.
\end{proof}

Let us add that there are $125$ sets of type $\mathcal{A}^p_{j_1}\times\mathcal{A}^q_{j_2}\times\mathcal{A}^r_{j_3}$. By the symmetry, using Theorem \ref{thm-criterion}, we are able to obtain all of them except two. The exceptions are $\mathcal{A}^p_0\times\mathcal{A}^q_0\times\mathcal{A}^r_0$ and $\mathcal{A}^p_4\times\mathcal{A}^q_4\times\mathcal{A}^r_4$. The next lemma justifies the lack of them in Table \ref{tab-oct} by proving that these products are empty.

\begin{lemma} \label{lem-R}
Exactly one or two of the following three inequalities
$$q^{-1}(p)+r^{-1}(p)>p,\quad r^{-1}(q)+p^{-1}(q)>q,\quad p^{-1}(r)+q^{-1}(r)>r$$
hold at the same time.
\end{lemma}

\begin{proof}
Summing the sides of the equality
$$\frac{q^{-1}(p)}p+\frac{p^{-1}(q)}q = 1+\frac1{pq}$$
with the sides of the symmetric equalities, we receive
$$\frac{q^{-1}(p)+r^{-1}(p)}p+\frac{r^{-1}(q)+p^{-1}(q)}q+\frac{p^{-1}(r)+q^{-1}(r)}r = 3+\frac1{qr}+\frac1{rp}+\frac1{pq}.$$
Hence the lemma follows.
\end{proof}

\section{A formula for $J_{pqr}$}

Before we present the announced formula, we need the following notation
$$\alpha_p=\min\{q^{-1}(p),r^{-1}(p),p-q^{-1}(p),p-r^{-1}(p)\},\quad \beta_p = (\alpha qr)^{-1}(p),$$
and similarly we define$\alpha_q$, $\alpha_r$, $\beta_q$, $\beta_r$. One can easily check that
$$\#\mathcal{A}^p_1=\#\mathcal{A}^p_3=\alpha_p, \quad \#\mathcal{A}^p_2=\beta_p-\alpha_p, \quad \#\mathcal{A}^p_0+\#\mathcal{A}^p_4=p-\alpha_p-\beta_p$$
and analogous inequalities hold for sets $\mathcal{A}_j^q$ and $\mathcal{A}_j^r$. Let also
$$\delta_p = \delta_{pq}\delta_{pr} + \delta_{rp}\delta_{qp},$$
similarly $\delta_q$ and $\delta_r$. If the first inequality from Lemma \ref{lem-R} is the only false or the only true one, then we put
$$R = \alpha_p(q-\alpha_q-\beta_q)(r-\alpha_r-\beta_r).$$
If the only true/false is the second or the third inequality, then we define $R$ analogously. In addition we put
\begin{align*}
S & = \sumc\delta_p\alpha_p(\beta_q-\alpha_q)(\beta_r-\alpha_r), \\
T & = \sump\delta_{qr}(\beta_p-\alpha_p)(\alpha_q\#\mathcal{A}^r_0 + \alpha_r\#\mathcal{A}^q_4),
\end{align*}
where
\begin{align*}
\sumc f(p,q,r) & = f(p,q,r)+f(r,p,q)+f(q,r,p), \\
\sump f(p,q,r) & = f(p,q,r)+f(r,p,q)+f(q,r,p) \\
& \quad +f(r,q,p)+f(p,r,q)+f(q,p,r).
\end{align*}

Now we are ready to present the main result of this section.

\begin{theorem} \label{thm-formula}
We have
$$J_{pqr} = R+S+T+\sumc\alpha_p\alpha_q(r-\alpha_r-\beta_r).$$
\end{theorem}

\begin{proof}
In order to make the notation more readable, we put
\begin{align*}
\pS{j_1j_2j_3} & = \sump \#\mathcal{A}^p_{j_1}\#\mathcal{A}^q_{j_2}\#\mathcal{A}^r_{j_3}, \\
\pS{j_1j_2j_3}(f(p,q,r)) & = \sump f(p,q,r)\#\mathcal{A}^p_{j_1}\#\mathcal{A}^q_{j_2}\#\mathcal{A}^r_{j_3}
\end{align*}
and analogously
\begin{align*}
\cS{j_1j_2j_3} & = \sumc \#\mathcal{A}^p_{j_1}\#\mathcal{A}^q_{j_2}\#\mathcal{A}^r_{j_3}, \\
\cS{j_1j_2j_3}(f(p,q,r)) & = \sumc f(p,q,r)\#\mathcal{A}^p_{j_1}\#\mathcal{A}^q_{j_2}\#\mathcal{A}^r_{j_3}
\end{align*}
By Theorem \ref{thm-criterion} we have
\begin{align*}
J_{pqr} & = \cS{001} + \cS{011} + \cS{334} + \cS{344} + \pS{123}(\delta_{pr}) \\
& \quad + \pS{012}(\delta_{qp}) + \pS{234}(\delta_{rq}) + \cS{122}(\delta_{pq}\delta_{pr}) + \cS{223}(\delta_{pr}\delta_{qr}).
\end{align*}
It is easy to observe that
$$\cS{001} + \cS{344} = \cS{100} + \cS{344} = \sumc \alpha_p\big(\#\mathcal{A}^q_0\#\mathcal{A}^r_0+\#\mathcal{A}^q_4\#\mathcal{A}^r_4\big) = R,$$
$$\cS{011} + \cS{334} = \cS{011} + \cS{411} = \sumc(p-\alpha_p-\beta_p)\alpha_q\alpha_r.$$
Now we consider sums containing $\delta$'s. The equalities above remain true, since if $\delta_{pr}+\delta_{rp}\neq1$, then the set $\mathcal{A}^q_2$ is empty. We have
$$\pS{123}(\delta_{pr}) = \cS{123}(\delta_{pr}) + \cS{321}(\delta_{rp}) = \sumc\alpha_p\alpha_q(\beta_r-\alpha_r),$$
$$\pS{012}(\delta_{qp}) + \pS{234}(\delta_{rq}) = \pS{210}(\delta_{qr}) + \pS{243}(\delta_{qr}) = T.$$
At last,
$$\cS{122}(\delta_{pq}\delta_{pr}) + \cS{223}(\delta_{pr}\delta_{qr}) = \cS{122}(\delta_{pq}\delta_{pr}) + \cS{322}(\delta_{rp}\delta_{qp}) = S.$$
By summing the received values, we get the thesis.
\end{proof}

As a consequence of Theorem \ref{thm-formula} we obtain Theorem \ref{thm-lb}.

\begin{proof}[Proof of Theorem \ref{thm-lb}]
We will use the fact that $ab\ge a+b-1$ for every positive integers $a$ and $b$. By Theorem \ref{thm-criterion} and the obvious inequality $R,S,T\ge0$, we have
\begin{align*}
J_{pqr} & \ge \sumc\alpha_p\alpha_q(r-2\alpha_r) = \frac12\sumc\alpha_p\big(\alpha_q(r-2\alpha_r)+(q-2\alpha_q)\alpha_r\big) \\
& \ge \frac12\sumc\big(\alpha_q+(r-2\alpha_r)-1+(q-2\alpha_q)+\alpha_r-1\big) \\
& = \frac12\sumc(q-\alpha_q-1+r-\alpha_r-1) \ge \frac12\sumc\big((q-1)/2+(r-1)/2\big)\\
& = (p-1)/2 + (q-1)/2 + (r-1)/2 > (p+q+r)/3 > \sqrt[3]{pqr},
\end{align*}
which completes the proof.
\end{proof}

\section{Polynomials with small $J_{pqr}$}

\begin{proof}
Let $q=tp-1$, where $3\le t<q^\varepsilon$ and let $r=2q+1=2tp-1$. Then it is not hard to verify that
$$\begin{array}{lll}
q^{-1}(p)=p-1, & r^{-1}(q)=1, & p^{-1}(r)=2t, \\
r^{-1}(p)=p-1, & p^{-1}(q)=t, & q^{-1}(r)=2tp-3, \\
\alpha_p=1, & \alpha_q=1, & \alpha_r=2, \\
\beta_p=1, & \beta_q=t, & \beta_r=2tp-2t-1, \\
\mathcal{A}^p_4=\emptyset, & \mathcal{A}^q_0=\emptyset, & \mathcal{A}^r_4=\emptyset,
\end{array}$$
and $\delta_{rp}=\delta_{pq}=1$. The remaining $\delta$'s from Theorem \ref{thm-formula} equal $0$. Thereby we have
$$\sumc\alpha_p\alpha_q(r-2\alpha_r) = (2tp-5) + 2(tp-3) + 2(p-2) < 6q$$
and
$$R= (p-\alpha_p-\beta_p)\alpha_q(r-\alpha_r-\beta_r) = (p-2)(2t-2) < 2q.$$
Since $\delta_p=\delta_q=\delta_r=0$, we have $S=0$. It remains to evaluate $T$:
\begin{align*}
T & = (\beta_q-\alpha_q)(\alpha_r\#\mathcal{A}^p_0 + \alpha_p\#\mathcal{A}^r_4) + (\beta_r-\alpha_r)(\alpha_p\#\mathcal{A}^q_0 + \alpha_q\#\mathcal{A}^p_4) \\
& = (t-1)2(p-2) < 2q.
\end{align*}
By Theorem \ref{thm-formula} we have $J_{pqr} < 10q$, while $pqr > q^{3-\varepsilon}$, so the proof is done.
\end{proof}

In a slightly more general class of the so-called inclusion-exclusion polynomials the exponent $1/3$ in Theorem \ref{thm-lb} is the best possible. We recall that
$$\Phi_{pqr}(x)=\frac{(1-x^{pqr})(1-x^p)(1-x^q)(1-x^r)}{(1-x^{qr})(1-x^{rp})(1-x^{pq})(1-x)}.$$
If we replace the assumptions that $p$, $q$, $r$ are primes by the assumption that they are pairwise coprime, then the formula above defines the inclusion-exclusion polynomial $Q_{\{p,q,r\}}$ (see \cite{Bachman-TernaryInEx}).

Let us denote by $J_{\{p,q,r\}}$ the number of jumping up coefficients of the the polynomial $Q_{\{p,q,r\}}$. As long as $p,q,r>2$, all results of our paper hold also for the polynomial $Q_{\{p,q,r\}}$.

The numbers $m$, $6m-1$, $12m-1$ are pairwise coprime for every positive integer $m$. Thus we can repeat the argument from the proof of Theorem \ref{thm-smallJ} to deduce that
$$J_{\{m,6m-1,12m-1\}} < 10(6m-1) < 15n^{1/3},$$
where $n=m(6m-1)(12m-1)$ and $m\ge3$. It derives infinitely many ternary inclusion-exclusion polynomials $Q_{\{p,q,r\}}$ for which $J_{\{p,q,r\}}<15n^{1/3}$, where $n=pqr$.

\section*{Acknowledgments}

The author would like to thank Wojciech Gajda for his remarks on this paper.

\end{document}